\documentclass{article}
\usepackage[utf8]{inputenc}

\usepackage{amsmath,color}
\usepackage{mathtools}
\usepackage{amsfonts}
\usepackage{amssymb}
\usepackage{graphicx}
\usepackage{hyperref}
\usepackage{amsthm}
\usepackage[capitalize]{cleveref}

\newcommand\numberthis{\addtocounter{equation}{1}\tag{\theequation}}

\newtheorem{theorem}{Theorem}
\newtheorem{lemma}[theorem]{Lemma}
\newtheorem{conjecture}[theorem]{Conjecture}
\newtheorem{proposition}[theorem]{Proposition}

\def\final{1}  

\ifnum\final=0  
\newcommand{\lnote}[1]{[{\small Luis: \bf #1}]}
\newcommand{\bnote}[1]{[{\small Brett: \bf #1}]}
\newcommand{\anonnote}[1]{[{\small anon: \bf #1}]}
\newcommand{\sidecomment}[1]{\marginpar{\tiny #1}}
\newcommand{\details}[1]{{\color{blue}\ [[#1]] }}
\else 
\newcommand{\lnote}[1]{}
\newcommand{\bnote}[1]{}
\newcommand{\anonnote}[1]{}
\newcommand{\sidecomment}[1]{}
\newcommand{\details}[1]{}
\fi  

\newcommand{\suchthat}{\mathrel{:}}
\newcommand{\giventhat}{\mid}

\newcommand{\RR}{\mathbb{R}}

\newcommand{\norm}[1]{{\lVert#1\rVert}}
\newcommand{\norms}[1]{\lVert#1\rVert^2}

\newcommand{\abs}[1]{\lvert#1\rvert}
\DeclarePairedDelimiter\vol{\lvert}{\rvert}

\newcommand{\aff}{\operatorname{aff}}

\newcommand{\inner}[2]{#1\cdot #2}
\newcommand{\goesto}{\rightarrow}

\newcommand{\var}{\operatorname{var}}
\newcommand{\cov}[2]{\operatorname{cov}({#1}, {#2})}
\newcommand{\prob}{\mathbb{P}}

\newcommand{\poly}{\operatorname{poly}}
\newcommand{\linspan}{\operatorname{span}}
\newcommand{\dist}{\operatorname{d}}
\newcommand{\ud}{\mathop{}\!\mathrm{d}}
\newcommand{\ind}{\operatorname{1}}

\newcommand{\pr}{\prob}
\newcommand{\e}{\operatorname{\mathbb{E}}}
\renewcommand{\cov}{\operatorname{cov}}
\newcommand{\ones}{\mathbf{1}}
\newcommand{\event}{E}
\newcommand{\lln}{\operatorname{lln}}

\title{Estranged facets and $k$-facets of Gaussian random point sets}
\author{Brett Leroux, Luis Rademacher}
\date{}

\begin{document}
\maketitle
\begin{abstract}
Gaussian random polytopes have received a lot of attention especially in the case where the dimension is fixed and the number of points goes to infinity. 
Our focus is on the less studied case where the dimension goes to infinity and the number of points is proportional to the dimension $d$.
 We study several natural quantities associated to Gaussian random polytopes in this setting. 
First, we show that the expected number of facets is equal to $C(\alpha)^{d+o(d)}$ where $C(\alpha)$ is some constant which depends on the constant of proportionality $\alpha$. 
We also extend this result to the expected number of $k$-facets. 
We then consider the more difficult problem of the asymptotics of the expected number of pairs of \emph{estranged facets} of a Gaussian random polytope. 
When $n=2d$, we determined the constant $C$ so that the expected number of pairs of estranged facets is equal to $C^{d+o(d)}$.
\end{abstract}

\section{Introduction}
A \emph{Gaussian random point set} is an i.i.d.\ sequence of standard Gaussian random points in $\RR^d$, i.e., each point in the set is distributed according to $N(0,I_d)$. The convex hull of a Gaussian random point set $\{X_1,\dotsc,X_n\}$ with $n$ samples is denoted $[X_1,\dotsc, X_n]$ and is called a \emph{Gaussian random polytope}.
In the study of random polytopes given as the convex hull of random points, many asymptotic results provide insight in the case where the dimension ($d$) is fixed but arbitrary and the number of points ($n$) grows.
For example, some of the basic results provide asymptotic expansions on the number of $j$-dimensional faces of a Gaussian random polytope for fixed $d$ and as $n \to \infty$ \cite{MR1149653, MR1254086,HMR,MR258089,MR156262}.
For the case where both the dimension and the number of points grow together, there are gaps in our understanding.
In this work we study this case. 
We provide asymptotic expansions of the expectation of several natural quantities associated to Gaussian random polytopes and, more generally, Gaussian random point sets.
The quantities we consider are: the number of facets, the number of $k$-facets, and the number of pairs of estranged facets.
We now recall the standard definitions of $k$-facets and estranged facets.

A \emph{$k$-facet} of a finite set of points $X \subset \RR^d$ in general position (namely, any subset of $d+1$ or less points is affinely independent) is a subset $\Delta \subset X$ of size $d$ such that the open halfspace on one side of $\aff \Delta$ contains exactly $k$ points from $X$. 
We use the notation $E_k(X)$ for the set of $k$-facets of $X$ and we define $e_k(X):=|E_k(X)|$. 
There is a long line of work on the \emph{$k$-facet problem} which asks one to determine the asymptotics of the maximum possible number of $k$-facets of a set of $n$ points in $\RR^d$ as a function of $n$, $k$ and $d$.
The first papers on the $k$-facet problem (\cite{lovasz1971number} and \cite{MR0363986}) only considered the case when the dimension is equal to two and even this case is still not well understood. 
See \cite{WagnerSurvey} for a survey on what is known. 
Although the majority of work on the $k$-facet problem is for deterministic points sets, the problem has also previously been studied for random point sets in \cite{BaranySteiger,clarkson,MR4376581}.

Let $P$ be a full-dimensional polytope.
We use the notation $f_j(P)$ for the number of $j$-dimensional faces of $P$. 
In particular, $f_{d-1}(P)$ is the number of facets. 
Note that if $X \subset \RR^d$ is a set of $n$ points in general position, then the $0$-facets of $X$ are precisely the facets of the polytope $P$ where $P$ is the convex hull of $X$ and so $f_{d-1}(P) = e_0(X)$ in this case.

A pair of facets of a polytope is called \emph{estranged} if they do not share any vertices (i.e., facets $F$ and $G$ are estranged if the set of vertices contained in $F$ is disjoint from the set of vertices contained in $G$). In this paper all polytopes we consider are simplicial with probability one. Under the standard polarity operation for polytopes, the polar of a simplicial polytope is a simple polytope and there is a one-to-one correspondence between pairs of estranged facets of the simplicial polytope $P$ and \emph{inner diagonals} of the polar dual $P^\ast$ of $P$. Here an \emph{inner diagonal} of a polytope is a line segment which joins two vertices of the polytope and that is contained, except for its endpoints, in the relative interior of the polytope. We clarify the motivation for studying estranged facets and inner diagonals in the next section. 

\subsection{Previous work and our contributions}

\paragraph{Expected number of facets and $k$-facets.}
Let $[X]$ denote the convex hull of $X$.

As mentioned above, for fixed dimension, an asymptotic formula for the expected number of facets of a Gaussian random polytope as the number of samples $n$ goes to infinity has been known for some time: It was shown in \cite{MR258089,MR156262} that for fixed $d \ge 2$, and a set $\{X_1,\dotsc, X_n\}$ of $n$ i.i.d.\ Gaussian random points in $\RR^d$,
\[
\mathbb{E}f_{d-1} ([X_1, \dotsc, X_n]) = \frac{2^d \pi^{\frac{d-1}{2}}}{\sqrt{d}}(\ln n)^{\frac{d-1}{2}}\bigl(1+o(1)\bigr) \text{ as } n \to \infty.
\]
Similar formulae are known for $\mathbb{E}f_{j} ([X_1, \dotsc, X_n])$ for $j=0, \dotsc , d$, see \cite{MR1149653, MR1254086,HMR}.

The above mentioned papers only address the case when the dimension is fixed and the number of samples goes to infinity.
More recently, progress has been made by B\"{o}r\"{o}czky, Lugosi and Reitzner in \cite{BLR} and Fleury in \cite{MR2875755} on the question of the asymptotic value of $\mathbb{E}f_{d-1} ([X_1, \dotsc, X_n])$ when both $d$ and $n$ are allowed to go to infinity. 
It is shown in \cite[Theorem 1.1]{BLR} that if $d\ge 78$ and $n \ge e^ed$, then 
\begin{align}\label{eq:o}
\mathbb{E}f_{d-1} ([X_1, \dotsc, X_n]) = 2^d\pi^{\frac{d-1}{2}}\sqrt{d}e^{\frac{d-1}{2}\lln\frac{n}{d}-\frac{d-1}{4}\frac{\text{lln}\frac{n}{d}}{\ln\frac{n}{d}}+(d-1)\frac{\Theta}{\ln\frac{n}{d}}+O(\sqrt{d}e^{-\frac{d}{10}})}
\end{align}
with $\Theta \in [-34,2]$ and $\lln=\log \log$.
Also, \cite[Theorem 1.3]{BLR} states that if $n-d = o(d)$, then
\begin{align}\label{eq:p}
\mathbb{E}f_{d-1} ([X_1, \dotsc, X_n]) = \binom{n}{d}\frac{1}{2^{n-d-1}}e^{\frac{1}{\pi}\frac{(n-d)^2}{d} + O\bigl(\frac{(n-d)^3}{d^2}\bigr)+o(1)}.
\end{align}

There are two gaps relevant to us in their expressions: 
(1) They only provide asymptotic expressions for $n-d = o(d)$ or $n\geq e^e d$. 
(2) For the case where $n$ grows proportional to $d$, they only establish exponential upper and lower bounds (with different bases of the exponential function in each bound). 
Our \cref{thm:k-facet} below fills in this missing piece. 
We show that when $n/d \to \alpha > 1$ and $k/(n-d) \to r \in [0,1]$ the expected number of $k$-facets grows like $C(\alpha,r)^{d+o(d)}$ where $C(\alpha,r)$ is a constant depending on $\alpha $ and $r$ and we provide a simple way to determine $C(\alpha,r)$ given $\alpha$ and $r$ (\cref{thm:k-facet}). 
Note that setting $k=0$ gives the asymptotics of the expected number of facets.

In \cite{spherical}, Bonnet and O'Reilly consider the convex hull of random points from the unit sphere in $\RR^d$. They call such polytopes \emph{spherical random polytopes} and they provide asymptotic expressions for the expected number of facets as $n$ and $d$ grow at different rates.
In the cases when $n-d= o(d)$ or $n/d \to \infty$, they obtain formulae for the expected number of facets of spherical random polytopes which match the corresponding formulae obtained in \cite{BLR} for the expected number of facets of Gaussian random polytopes, i.e.\ equations \eqref{eq:o} and \eqref{eq:p} above. 
Such a correspondence is not particularly surprising given the fact that Gaussian random points concentrate around a thin spherical shell of radius $\sqrt{d}$ in high dimension. 
Our result shows that this correspondence continues for the case when $n$ is proportional to $d$: for any $\alpha>1$, \cref{thm:k-facet} says that the expected number of facets of a Gaussian random polytope with $n\sim \alpha d$ vertices is equal to $C(\alpha)^{d+o(d)}$ for some constant $C(\alpha)$. 
For spherical random polytopes, the case when the number of vertices is equal to $n\sim \alpha d$ for some $\alpha>1$ is dealt with in \cite[Theorem 4.2]{spherical}. 
The asymptotic formula given there is also of the form $C(\alpha)^{d+o(d)}$ for some constant $C(\alpha)$. 
Some algebra shows that the constants are the same in both the spherical and Gaussian random cases.

\paragraph{A formula from \cite{HMR} extended to $k$-facets.}
\cite[Theorem 3.2]{HMR} provides a formula that expresses the probability that a fixed subset of $d$ out of $n$ Gaussian random points form a facet of the convex hull of the whole set. The formula turns the original probability involving $n$ random vectors in $\RR^d$ into a simpler probability involving $n-d+1$ real valued random variables.
Their proof is an application of the affine Blaschke-Petkantschin formula.

We extend the formula to the case of $k$-facets (\cref{thm:3.2}). Our proof does not use the Blaschke-Petkantschin formula and is based on a slightly different probabilistic argument.

\paragraph{Expected number of pairs of estranged facets.}
We show in \cref{thm:estranged} that if $X$ is a set of $2d$ i.i.d.\ Gaussian random points in $\RR^d$, then the expected number of pairs of estranged facets of $[X]$ is equal to $C^{d+o(d)}$ where $C\approx 1.7696$.

The main technique in the proof is the affine Blaschke-Petkantschin formula applied twice on a partition of the $2d$ points into two $d$-subsets to express the probability that they are facets simultaneously.
This is combined with known estimates of the expected volume of a random simplex (one of the main terms in the affine Blaschke-Petkantschin formula) and a simple asymptotic expansion of integrals (\cref{prop:asymptotic}, see below).

To put this result in context, we recall the following conjecture of von Stengel \cite{Vs}:
\begin{conjecture}[\cite{Vs}]
The maximum number of pairs of estranged facets of any simplicial $d$-polytope with $2d$ vertices is $2^{d-1}$, which is attained by the $d$-dimensional cross polytope. 
\end{conjecture}

Although von Stengel's conjecture is still open, a number of similar questions about estranged facets (and their polar equivalent, inner diagonals) were answered by Bremner and Klee \cite{MR1698257} who argue that estranged facets are worthy of more study given that they are an intrinsically interesting combinatorial feature of convex polytopes.

Aside from their intrinsic interest, estranged facets are also relevant to the study of Nash equilibria of bimatrix
games \cite{Vs}. 
Indeed, this was the original context for the above conjecture of von Stengel. 
Although estranged facets themselves do not directly correspond to any particular quantity of interest in bimatrix games, they have been used by B\'{a}r\'{a}ny, Vempala and Vetta \cite{BVV} in the analysis of a Las Vegas algorithm for finding Nash equilibria in bimatrix games. 
In particular, their analysis required them to determine concentration bounds for the number of Nash equilibria in random games. 
This in turn required them to prove an upper bound on the expected number of pairs of estranged facets of a random polytope whose vertices are either i.i.d.\ Gaussian or uniform in the $d$-cube \cite[Lemma 13]{BVV}. 
In contrast to our \cref{thm:estranged}, \cite[Lemma 13]{BVV} is only meaningful in the case when the dimension $d$ is fixed and the number of points $n$ goes to infinity. 

Finally, we remark that estranged facets are also relevant to the study of the diameter problem for convex polytopes, i.e., the question of the maximum diameter of the graph of a simple $d$-polytope with $n$ facets.
As previously mentioned, estranged facets of a simplicial polytope correspond, via the polar operation, to inner diagonals of a simple polytope.
It has been shown that the pair of vertices which attains the maximum distance in the graph of a simple polytope must be the endpoints of some inner diagonal of the polytope \cite{MR206823}.

\paragraph{Simple asymptotic expansion of integrals.}
Our asymptotic expansions of expected values are based on the formula $\int_{\RR^d} f(x)^p \ud x = \norm{f}_\infty^{p+o(p)}$, stated formally as \cref{prop:asymptotic}. 
This is a simple result that provides asymptotic expansions of integrals that follows immediately from the known fact that the $L^p$ norm of a function converges to the $L^\infty$ norm as $p \to \infty$ under mild assumptions (\cref{prop:limit}).

\subsection{Outline of the paper.}
\cref{sec:prelim} introduces notation and collects some propositions that will be used later including a result about the expected volume of a Gaussian simplex as well as a result about asymptotic expansions of integrals based on $L^p$ norms. 
In \cref{sec:k-facet} we establish our asymptotic formula for the expected number of $k$-facets of a Gaussian random polytope. 
Finally, \cref{sec:estranged} establishes the asymptotic formula for the expected number of estranged facets of a Gaussian random polytope with $2d$ vertices.

\section{Preliminaries}\label{sec:prelim}

Let $f_X$ denote the PDF of random variable $X$.
Let $\e_X\bigl(f(X,Y)\bigr)$ denote the expectation with respect to $X$ only, and similarly for $\pr_X$. Namely, $\e_X\bigl(f(X,Y)\bigr) = \e\bigl(f(X,Y) \bigm| Y \bigr)$.
For a random vector $X$, let $\cov(X)$ denote the covariance matrix of $X$.
Asymptotic notation $f(d) \sim g(d)$ means $f(d)/g(d) \to 1$ as $d \to \infty$.
For a set $A=\{\ldots \}$ in a measurable space, let $\ind A = \ind \{ \ldots \}$ denote the indicator function of $A$.
For a measurable set $K \subseteq \RR^d$, let $\vol{K}$ denote the volume of $K$. 
Let $[X]$ denote the convex hull of $X$.
\begin{proposition}[{Blaschke's formula, \cite[Proposition 3.5.5]{MR3185453} \cite[Lemma 4]{MR2891161}}]\label{prop:blaschke}
Let $X_1, \dotsc, X_{d+1}$ be i.i.d.\ $d$-dimensional random vectors with finite second moment. 
Then
\[
\det \cov (X_1) = \frac{d!}{d+1} \e\bigl(\vol[\big]{[X_1,\dotsc,X_{d+1}]}^2 \bigr).
\]
\end{proposition}
\begin{proof}
\cite[Lemma 4]{MR2891161} states and proves the claim for the uniform distribution in a convex body. That proof works essentially unchanged for any distribution with finite second moment.
\end{proof}

We will need the following well-known result about the expected volume of a Gaussian simplex.
See e.g. \cite[p. 377]{10.2307/1426176}. 
\begin{proposition}\label{prop:gaussiansimplex}
Let $X_1, \dotsc, X_{d+1}$ be i.i.d.\ $d$-dimensional Gaussian random vectors. Then
\[
\e\bigl(\vol[\big]{[X_1,\dotsc,X_{d+1}]}\bigr) = \frac{\sqrt{d+1}}{2^{d/2}\Gamma(\frac{d}{2}+1)} \sim \frac{1}{\sqrt{\pi}}\left(\frac{e}{d}\right)^{d/2}.
\]
\end{proposition}

We use the following asymptotic approximation of integrals: $\int_{\RR^d} f(x)^p \ud x = \norm{f}_\infty^{p+o(p)}$ (\cref{prop:asymptotic}). 
It follows easily from the fact that the $L^p$ norm converges to the $L^\infty$ norm as $p \to \infty$ under mild assumptions (\cref{prop:limit}).
\begin{proposition}[{\cite[p. 71]{MR924157}}]\label{prop:limit}
Let $1\leq q < \infty$. 
Let $f \in L^\infty(\RR^d) \cap L^q(\RR^d)$.
Then $\norm{f}_\infty = \lim_{p \to \infty} \norm{f}_p$.
\end{proposition}

\begin{proposition}\label{prop:asymptotic}
Let $1\leq q < \infty$. 
Let $f \in L^\infty(\RR^d) \cap L^q(\RR^d)$ and assume that $f$ is nonnegative and $C:= \norm{f}_\infty \neq 1$.
Then, as $p \to \infty$, $\int_{\RR^d} f(x)^p \ud x = C^{p+o(p)}$ (where $o(p)$ can depend on $f$).
\end{proposition}
\begin{proof}
Let $a_p = \int_{\RR^d} f(x)^p \ud x$.
From \cref{prop:limit} we have $\lim_{p \goesto \infty} a_p^{1/p} = C$.
Write $a_p = C^{p + g(p)}$ for some function $g$.

To conclude, we will now show that $g(p) = o(p)$.
Note that $a_p^{1/p} = C^{1+\frac{g(p)}{p}}$, so that, applying $\lim_{p \to \infty}$ to both sides we get $\lim_{p \to \infty} C^{\frac{g(p)}{p}} = 1$, which implies $\lim_{p \to \infty} \frac{g(p)}{p} = 0$.
\end{proof}

We need the following known inequality (the constant has not been optimized). 
\begin{lemma}\label{lem:k}
If $X$ is a (real valued) mean zero logconcave random variable then $\e(\abs{X}) \geq \frac{1}{8} \sqrt{\e(X^2) }$.
\end{lemma}
\lnote{can one remove ``mean zero''}
\begin{proof}
The inequality is invariant under scaling and therefore it is enough to prove it when $X$ is isotropic (i.e.\ when $\e(X^2)=1$).
It is known \cite[Lemma 5.5]{LV07} that the density of an isotropic logconcave random variable is at most 1.
Therefore, using Markov's inequality, $ 1/2 \leq \pr(\abs{X} \geq 1/4) \leq 4 \e(\abs{X})$. 
The claim follows.
\end{proof}

\section{Facets and \texorpdfstring{$k$}{k}-facets}\label{sec:k-facet}

In this section we study the expected number of $k$-facets of Gaussian random polytopes. We give an asymptotic formula for the expected number of $k$-facets in the case when the dimension $d$ goes to infinity and the number of samples $n$ grows linearly with $d$.

Before establishing our asymptotic formula, we need to establish the following result which reduces the problem of computing $\mathbb{E}f_{d-1} ([X_1, \dotsc, X_n])$ from a $d$-dimensional problem to a $1$-dimensional problem.
\begin{theorem}\label{thm:3.2}
Let $X_1,\dotsc, X_n$ be $n \ge d+1$ i.i.d.\ standard Gaussian random vectors in $\RR^d$. 
Then the expected number of $k$-facets of $\{X_1, \dotsc, X_n\}$ is equal to 
\[
\binom{n}{d}\mathbb{P}\bigl(Y \in E_k(\{Y,Y_1, \dotsc, Y_{n-d}\}) \bigr)
\]
where $Y$ is $N(0, \frac{1}{d})$, $Y_i$ is $N(0,1)$ for $i = 1, \dotsc ,n-d$ and  $Y, Y_1, \dotsc , Y_{n-d}$ are independent. 
\end{theorem}
\begin{proof}
By linearity of expectation and symmetry, it is enough to show that the probability that $\{X_1, \dotsc, X_d\}$ is a $k$-facet is $\pr \bigl(Y \in E_k(\{Y,Y_1, \dotsc, Y_{n-d}\}) \bigr) $.

Let $V$ be a random unit vector perpendicular to $\aff\{ X_1, \dotsc, X_d\}$ but with its orientation (sign) chosen independently at random among the two choices.
Define $Y = V\cdot X_1$ and $Y_i = V \cdot X_{i+d}$, $i=1,\dotsc, n-d$.
Using that $V$ is independent of $X_{d+1},\dotsc, X_{n}$, it is clear that the $Y_i$s are i.i.d.\ $N(0,1)$.
Moreover, notice that by symmetry, the distribution of $V$ conditioned on $Y$ is still uniform on the unit sphere.
That is, $V$ is independent of $Y$, which implies that $Y$ is independent of $Y_1, \dotsc, Y_{n-d}$.

We now determine the distribution of $Y$. 
Note that $Y^2$ is the squared distance of $\aff\{ X_1, \dotsc, X_d\}$ to the origin, which is given by $1/\norms{A^{-1} \ones}$, where $A$ is the matrix having $X_1, \dotsc, X_d$ as rows. 
By the invariance under orthogonal transformations of the distribution of $A$, the distribution of $A^{-1}$ is also invariant under orthogonal transformations and the distribution of $1/\norms{A^{-1} \ones}$ is the same as the distribution of $\frac{1}{d\norms{A^{-1} e_1}}$, where $\frac{1}{\norms{A^{-1} e_1}}$ is the squared distance between $X_1$ and $\linspan \{X_2, \dotsc, X_d\}$. 
This is distributed as $\chi^2_1$ (namely, $N(0,1)$ squared). 
Thus, using the random sign of $V$, the distribution of $Y$ is $N(0,1/d)$.

In summary, $Y$ and $Y_i$s are distributed as in the statement.
Moreover, the event that $\{X_1, \dotsc, X_d\}$ is a $k$-facet of $\{X_1, \dotsc, X_n\}$ is the same as the event that $Y$ is a $k$-facet of $ \{Y,Y_1, \dotsc, Y_{n-d}\} $.
\end{proof}

We remark that \cref{thm:3.2} is heavily inspired by the work of Hug, Munsonius and Reitzner in \cite{HMR}.
In particular, \cref{thm:3.2} is a simple generalization of \cite[Theorem 3.2]{HMR} from facets to $k$-facets. See \cite[Theorem 3.2]{HMR} for an alternative proof of the above theorem (in the case of facets) using the affine Blaschke-Petkantschin formula. 

We are know ready to state our main result on facets/$k$-facets of Gaussian random polytopes. We use the notation 
\[
\Phi(y) := \frac{1}{\sqrt{2 \pi}} \int_{-\infty}^y e^{-s^2/2} \ud s , \text{  and } \phi(y):=\Phi'(y) = \frac{1}{\sqrt{2\pi}}e^{-y^2/2} 
\]
for the CDF and PDF of the standard Gaussian distribution.

\begin{theorem}\label{thm:k-facet}
Fix $\alpha>1$, and $r \in [0,1]$ and assume that $n/d \to \alpha$ as $d \to \infty$ and that $k/(n-d) \to r$ as $d \to \infty$. Let $X$ be a set of $n$ i.i.d.\ Gaussian random points in $\RR^d$. Then the expected number of $k$-facets of $X$ is equal to 
\[
\left(2^{\alpha H(\frac{1}{\alpha})} 2^{(\alpha-1) H(r)} \sqrt{2\pi} c_{\alpha,r}\right)^{d+o(d)} \text{ as } d \to \infty,
\]
where 
\[
c_{\alpha,r} := \max_{y \in \RR} \{ \Phi(y)^{r\alpha} (1-\Phi(y))^{\alpha-1-r\alpha} \phi(y)\}.
\]
and $H(r)$ is the binary entropy function. The rate of convergence in the above $o(d)$ is not universal as it depends on $\alpha$ and $r$ and on the rate of convergence of $n/d$ to $\alpha$ and $k/(n-d)$ to $r$.
\end{theorem}
\begin{proof}
From \cref{thm:3.2}, $\mathbb{E} e_k(X) = \binom{n}{d}\mathbb{P}\big(Y \in E_k(\{Y,Y_1, \dotsc, Y_{n-d}\}) \big)$ where $Y$ is $N(0, \frac{1}{d})$, $Y_i$ is $N(0,1)$ for $i = 1, \dotsc ,n-d$ and  $Y, Y_1, \dotsc , Y_{n-d}$ are independent. Notice that if $k\neq\frac{n-d}{2}$, then 
\[
\mathbb{P}\big(Y \in E_k(\{Y,Y_1, \dotsc, Y_{n-d}\}) \big) = 2\binom{n-d}{k}\frac{\sqrt{d}}{\sqrt{2\pi}} \int\limits_{-\infty}^{\infty} \Phi(y)^k \bigl(1-\Phi(y)\bigr)^{n-d-k} e^{\frac{-dy^2}{2}} \ud y.
\]
If $k = \frac{n-d}{2}$, the above formula counts each potential $k$-facet twice, because in this case each side of the hyperplane represented by $Y$ could contain exactly $\frac{n-d}{2}$ points. Therefore, if $k = \frac{n-d}{2}$, the above formula holds after removing the factor of two on the right-hand side. This factor of two is not important for our result, and we have that
\begin{align*}
   \mathbb{E} e_k(X)
    &= \Theta(1)\binom{n}{d}\mathbb{P}\big(Y \in E_k(\{Y,Y_1, \dotsc, Y_{n-d}\}) \big) \\
    &= \Theta(1)\binom{n}{d}\binom{n-d}{k}\frac{\sqrt{d}}{\sqrt{2\pi}} \int\limits_{-\infty}^{\infty} \Phi(y)^k \bigl(1-\Phi(y)\bigr)^{n-d-k} e^{-dy^2/2} \ud y \\
    &= \Theta(1)\binom{n}{d}\binom{n-d}{k}\sqrt{d}(2\pi)^{\frac{d-1}{2}} \int\limits_{-\infty}^{\infty} \Phi(y)^k \bigl(1-\Phi(y)\bigr)^{n-d-k} \phi(y)^d \ud y.
\end{align*}
We will use \cref{prop:asymptotic} to estimate the integral in the above expression. 
In particular, we will show that the integral is equal to $c_{\alpha,r}^{d+o(d)}$ where $c_{\alpha,r}:=\|f\|_\infty$ and $f(y) := \Phi(y)^{r(\alpha-1)} \bigl(1-\Phi(y)\bigr)^{(1-r)(\alpha-1)} \phi(y)$. 
In order to establish this estimate, we first need to restrict the integral to some finite interval, the length of which does not depend on $d$ but does depend on $\alpha,r$. 
In order to accomplish this, first observe that we can upper bound the terms in front of the integral by $\binom{n}{d}\binom{n-d}{k}\sqrt{d}(2\pi)^{\frac{d-1}{2}} = O\bigl(2^n2^n(2\pi)^{\frac{d-1}{2}}\bigr)=O\bigl((4^{\alpha}\sqrt{2\pi})^d \bigr)$. 
Now choose $R(\alpha)$ so that $\phi\bigl(R(\alpha)\bigr)<\frac{1}{4^\alpha\sqrt{2\pi}}$. 
For technical reasons, we also need to assume that our region of integration is big enough so that it contains some $y_0 \in \RR$ so that $c_{\alpha,r} =  f(y_0)$\details{$y_0$ is not at $\infty$ because $\lim_{y \to \pm\infty}f(y) = 0$}. 
So choose $R(\alpha,r)$ so that $R(\alpha,r)\ge R(\alpha)$ and so that $[-R(\alpha,r),R(\alpha,r)]$ contains some $y_0$ as above. 
Using the fact that $\Phi(y)^k \bigl(1-\Phi(y)\bigr)^{n-d-k}<1$ and $\phi\bigl(R(\alpha,r)\bigr)<\frac{1}{4^\alpha\sqrt{2\pi}}$, we know that the right tail of the integral is upper bounded by
\begin{align*}
\int\limits_{R(\alpha,r)}^{\infty} \Phi(y)^k \bigl(1-\Phi(y)\bigr)^{n-d-k} \phi(y)^d \ud y 
&\le \int\limits_{R(\alpha,r)}^{\infty}   \phi(y)^{d-1} \phi(y) \ud y\\
& \le \bigg(\frac{1}{4^\alpha \sqrt{2\pi}}\bigg)^{d-1}\int\limits_{R(\alpha,r)}^{\infty}    \phi(y) \ud y \\
&=O((4^\alpha \sqrt{2\pi})^{-d})
\end{align*}
and the same estimate holds for the left tail. Therefore,
\begin{align*}
   \mathbb{E} e_k(X) &= \Theta(1)\binom{n}{d}\binom{n-d}{k}\sqrt{d}(2\pi)^{\frac{d-1}{2}} \int\limits_{-\infty}^{\infty} \Phi(y)^k \bigl(1-\Phi(y)\bigr)^{n-d-k} \phi(y)^d \ud y\\
   &= \Theta(1)\binom{n}{d}\binom{n-d}{k}\sqrt{d}(2\pi)^{\frac{d-1}{2}} \int\limits_{-R(\alpha,r)}^{R(\alpha,r)} \Phi(y)^k \bigl(1-\Phi(y)\bigr)^{n-d-k} \phi(y)^d \ud y +O(1) \\
   & = \Theta(1)\binom{n}{d}\binom{n-d}{k}\sqrt{d}(2\pi)^{\frac{d-1}{2}} \int\limits_{-R(\alpha,r)}^{R(\alpha,r)} \Phi(y)^k \bigl(1-\Phi(y)\bigr)^{n-d-k} \phi(y)^d \ud y 
\end{align*}
where the last equality uses the fact that $\mathbb{E}e_k(X)\ge1$ so that the $O(1)$ term can be absorbed into the $\Theta(1)$ factor in front.

Now for $y \in [-R(\alpha,r),R(\alpha,r)]$, $\Phi(y)$ and $1-\Phi(y)$ both take values in some fixed interval, i.e.\ $\Phi(y) = \Theta(1)$ and $1-\Phi(y) = \Theta(1)$. 
Recall that we are assuming that $n/d \to \alpha $ and $k/(n-d)\to r$ as $d \to \infty$ which means that $n = \alpha d +o(d)$ and $k = r(\alpha-1)d+o(d)$ and therefore that $n-d-k = (\alpha-1 )d-r(\alpha-1)d+o(d)$. 
This means that $\Phi(y)^k =\Phi(y)^{r(\alpha-1)d}\Theta(1)^{o(d)} =e^{o(d)} \Phi(y)^{r(\alpha-1)d}$ and that $\bigl(1-\Phi(y)\bigr)^{n-d-k} = \bigl(1-\Phi(y)\bigr)^{(\alpha-1 )d-r(\alpha-1)d}\Theta(1)^{o(d)} = e^{o(d)}\bigl(1-\Phi(y)\bigr)^{(\alpha-1 )d-r(\alpha-1)d}$ for $y \in [-R(\alpha,r),R(\alpha,r)]$. 
Therefore we have shown that 
\begin{align*}
    \int\limits_{-R(\alpha,r)}^{R(\alpha,r)} &\Phi(y)^k \bigl(1-\Phi(y)\bigr)^{n-d-k} \phi(y)^d \ud y\\
    &= e^{o(d)}\int\limits_{-R(\alpha,r)}^{R(\alpha,r)}\Phi(y)^{r(\alpha-1)d}\bigl(1-\Phi(y)\bigr)^{(1-r)(\alpha-1 )d} \phi(y)^d \ud y.
\end{align*}
Let $\hat{f}:=f\cdot\ind \{ -R(\alpha,r)<y<R(\alpha,r)\} $. 
Define $\hat{c}_{\alpha,r} : = \|\hat{f} \|_{\infty}$. 
Recall that we are assuming that $f$ attains its maximum somewhere in the interval $[-R(\alpha,r),R(\alpha,r)]$ so we have that $\hat{c}_{\alpha,r} = c_{\alpha,r}$.

By \cref{prop:asymptotic}, we have that 
\[
\int\limits_{-R(\alpha,r)}^{R(\alpha,r)}\Phi(y)^{r(\alpha-1)d}\bigl(1-\Phi(y)\bigr)^{(1-r)(\alpha-1 )d} \phi(y)^d \ud y = (\hat{c}_{\alpha,r})^{d+o(d)}=(c_{\alpha,r})^{d+o(d)}.
\]
Combining everything,
\begin{align*}
   \mathbb{E} e_k(X)
    &=\Theta(1)\binom{n}{d}\binom{n-d}{k}\sqrt{d}(2\pi)^{\frac{d-1}{2}} \int\limits_{-R(\alpha,r)}^{R(\alpha,r)} \Phi(y)^k \bigl(1-\Phi(y)\bigr)^{n-d-k} \phi(y)^d \ud y \\
    & = \Theta(1)\binom{n}{d}\binom{n-d}{k}\sqrt{d}(2\pi)^{\frac{d-1}{2}}e^{o(d)}(c_{\alpha,r})^{d+o(d)} \\
    &= \left(2^{\alpha H(\frac{1}{\alpha})} 2^{(\alpha-1) H(r)} \sqrt{2\pi} c_{\alpha,r}\right)^{d+o(d)}.
\end{align*}
For the last step one can use the Stirling approximation of the Gamma function and the fact that Gamma is continuous on $\RR_+$ to obtain the asymptotic estimates of the binomial coefficients. \details{Using Stirling approx. of Gamma and the fact that Gamma is continuous on $\RR_+$,
\begin{align*}
    \binom{n-d}{k}  &= \frac{\Gamma(n-d+1)}{\Gamma(k+1)\Gamma(n-d-k+1)} \\
    & \sim \frac{\Gamma(d(\alpha-1)+1)}{\Gamma(dr(\alpha-1)+1)\Gamma(d(1-r)(\alpha-1)+1)}\\
    &\sim \frac{\sqrt{d(\alpha-1)}}{\sqrt{dr(\alpha-1)}\sqrt{d(\alpha-1)(1-r)}\sqrt{2\pi}}\frac{(d(\alpha-1))^{d(\alpha-1)}}{(dr(\alpha-1))^{dr(\alpha-1)}(d(\alpha-1)(1-r))^{d(\alpha-1)(1-r)}} \\
    & =\frac{1}{\sqrt{r}\sqrt{2\pi d(\alpha-1)(1-r)}} \frac{1}{r^{dr(\alpha-1)}(1-r)^{d(\alpha-1)(1-r))}}
\end{align*}
Can ignore the first term (it's $2^{o(d)}$). Taking $log_2$ of second, get
\begin{align*}
    -dr(\alpha-1)\log_2r - d(\alpha-1)(1-r)\log_2 (1-r) 
    &= d(\alpha-1)H(r)\\
\end{align*}
So $\binom{n-d}{k} = 2^{o(d)}2^{d(\alpha-1)H(r)}$
}
\end{proof}


\section{Estranged facets}\label{sec:estranged}

We say that two facets of a polytope are \emph{estranged} if they do not share any vertices.
The main result of this section is \cref{thm:estranged}, which gives an asymptotic estimate of the expected number of estranged facets of the convex hull of $2d$ Gaussian random points in $\RR^d$.
\begin{theorem}\label{thm:estranged}
Let $X$ be a set of $2d$ i.i.d.\ Gaussian random points in $\RR^d$.
Let $N$ be the number of (unordered) pairs of estranged facets in $[X]$.
Then
\[
\e(N) = (4C_{\ref*{lem:pair}})^{d+o(d)},
\]
where $C_{\ref*{lem:pair}} \in (0, 1/2)$ is the universal constant from \cref{lem:pair}.
\end{theorem}

Our proof uses the affine Blaschke-Petkantschin formula \cite[Theorem 7.2.7]{schneider}, a change of variable formula that involves the volume of a random simplex.
We will need the following estimate of the volume of a random simplex in a halfspace:
\begin{lemma}\label{lem:lowerbound}
Let $H \subseteq \RR^{d-1}$ be a halfspace that contains the origin.
Let $Z \in \RR^{(d-1) \times d}$ be a random matrix with i.i.d.\ standard Gaussian entries truncated to be in $H^d$.
Then
\[
\e\bigl( \vol{Z} \bigr) 
\geq \sqrt{1-\frac{2}{\pi}} \frac{\sqrt{d}}{2^{\frac{d+5}{2}}\Gamma(\frac{d+1}{2})} 
= \left(\frac{e}{d}\right)^{d/2} 2^{o(d)}
\]
(where $o(d)$ does not depend on $H$ and using abbreviated notation $\vol{Z} = \vol[\big]{[Z_1, \dotsc, Z_d]}$).
\end{lemma}
\begin{proof}
The idea of the proof is to compare $Z$ with the Gaussian case (namely, without truncation). It is easier to do this for the second moment instead of the first, and one can relate the first and the second moments via Jensen's inequality and a suitable reverse for our case, \cref{lem:k}. 

By applying a rotation it is enough to prove for $H = \{x \in \RR^{d-1} \suchthat x_1 \leq t\}$ with $t\geq 0$.
Let $W$ be $Z$ with a row of ones appended. 
Then 
\begin{equation}\label{eq:detvol}
\vol{Z}/d = \abs{\det(W)}/d!.
\end{equation}
That is, $\vol{Z} = \abs{\det(W)}/(d-1)!$.
Let $W_1, \dotsc, W_d$ be the rows of $W$.
Let $A = \{ x \in \RR^d \suchthat (\forall i) x_i \leq t \}$.
Note that $W_1$ is distributed as standard Gaussian truncated to $A$.
We have $\abs{\det(W)} = \prod_{i=1}^d \dist(W_i, \linspan W_{(i+1)...d})$ (where $\dist(\cdot, \cdot)$ denotes point-subspace distance) and
\begin{align}
\e\bigl( \abs{\det(W)}\bigr)
&= \e\Bigl( \dist(W_1, \linspan W_{2...d}) \prod_{i=2}^d \dist(W_i, \linspan W_{(i+1)...d}) \Bigr) \label{eq:det}\\
&= \e\Bigl( \e\bigl(\dist(W_1, \linspan W_{2...d}) \bigm| W_{2...d} \bigr) \prod_{i=2}^d \dist(W_i, \linspan W_{(i+1)...d}) \Bigr).\nonumber
\end{align}
\details{total expectation + pull out known factors}

Let $v \in \RR^d$ be such that $\sum_{i=1}^d v_i = 0$ and $\norm{v} = 1$. 
Using \cref{lem:k}, $\e (v^T W_1) = 0$, and the fact that the variance of a Gaussian truncated to $(-\infty, t]$ with $t\geq 0$ is at least $1- 2/\pi$
we get
\begin{align*}
\e( \abs{v^T W_1} ) \geq \frac{1}{8} \sqrt{\e((v^T W_1)^2)} = \frac{1}{8} \sqrt{\var(v^T W_1)} \geq \frac{1}{8} \sqrt{1-\frac{2}{\pi}} := c'.
\end{align*}

Now, to express $ \dist(W_1, \linspan W_{2...d})$, let $V$ be a random vector that is a unit vector normal to $\linspan W_{2...d}$ (sign will not matter) and let $W_1'$ be an independent standard Gaussian in $\RR^d$. 
We have the following comparison inequality between $W_1$ (truncated Gaussian) and $W_1'$ (not truncated), using moment inequalities and the fact that, conditionioning on $W_{2...d}$, vector $V$ is a fixed unit vector perpedicular to the all ones vector $W_d$ so that our analysis for $v$ above applies:
\begin{align*}
\e\bigl( \dist(W_1, \linspan W_{2...d}) \bigm| W_{2...d} \bigr) 
&= \e\bigl( \abs{V^T W_1} \bigm| W_{2...d} \bigr) \\
&\geq c' \\
&= c' \sqrt{ \e\bigl( \dist(W_1', \linspan W_{2...d})^2 \bigm| W_{2...d} \bigr) } \\
&\geq c' \e \bigl( \dist(W_1', \linspan W_{2...d}) \bigm| W_{2...d} \bigr). 
\end{align*}
This in \eqref{eq:det} implies, defining $W'$ as $W$ with the first row $W_1$ substituted by $W_1'$:
\begin{align*}
\e( \abs{\det(W)}) 
&\geq c' \e\Bigl( \e\bigl(\dist(W_1', \linspan W_{2...d}) \bigm| W_{2...d} \bigr) \prod_{i=2}^d \dist(W_i, \linspan W_{(i+1)...d}) \Bigr) \\
&= c' \e\Bigl( \dist(W_1', \linspan W_{2...d}) \prod_{i=2}^d \dist(W_i, \linspan W_{(i+1)...d} ) \Bigr) \\
&= c' \e(  \abs{\det(W')}) \\
&= c' \frac{(d-1)! \sqrt{d}}{2^{\frac{d-1}{2}}\Gamma(\frac{d+1}{2})} \quad \text{(using \cref{prop:gaussiansimplex} and the idea in \cref{eq:detvol})}.
\end{align*}
Thus
\begin{align*}
\e\bigl( \vol{Z} \bigr) 
&= \frac{\e\bigl( \abs{\det(W)}\bigr)}{(d-1)!}  
\geq \frac{c' \sqrt{d}}{2^{\frac{d-1}{2}}\Gamma(\frac{d+1}{2})}. \details{\approx \poly(d) (e/d)^{d/2}}\qedhere
\end{align*}
\end{proof}

We will now complete the proof of \cref{thm:estranged}. 
Most of the proof is in the following lemma (\cref{lem:pair}), which estimates the probability that a fixed partition of the random points is a pair of facets. \Cref{thm:estranged} then follows by linearity of expectation.
The proof of \cref{lem:pair} is somewhat similar to the proof of \cite[Theorem 1.3]{HR} which gives an upper bound for the
variance of the number of facets of a Gaussian random polytope in the case where the dimension is fixed and the number of points increases. The main difficulty in the proof of both \cite[Theorem 1.3]{HR} and \cref{lem:pair} is to prove an upper bound for the probability that given pair of subsets of vertices are both facets of the polytope. In contrast to \cite[Theorem 1.3]{HR}, our \cref{lem:pair} is meaningful when the dimension increases with the number of points. However, \cref{lem:pair} does not give any bound on the variance because we only consider pairs of facets with no points in common.

Let $F(P)$ be the set of facets (as a family of subsets of vertices) of polytope $P$.
\begin{lemma}\label{lem:pair}
Let $X, Y$ be two independent sets of $d$ i.i.d.\ Gaussian random points in $\RR^d$.
Then
\[
\pr\bigl(X, Y \in F([X,Y])\bigr) = C_{\ref*{lem:pair}}^{d+o(d)},
\]
where
\[
C_{\ref*{lem:pair}} :=
\sup_{\substack{\rho \geq 0 \\ w \in [-1,1]}} e^{-\rho^2}\Phi\left(\frac{\rho(1-w)}{\sqrt{1-w^2}}\right)^2 \sqrt{1-w^2} \approx 0.4424.
\]
\end{lemma}
\begin{proof}
Let $H(\rho, \theta) = \{x \in \RR^d \suchthat \theta \cdot x = \rho \}$, $ H_-(\rho, \theta) = \{x \in \RR^d \suchthat \theta \cdot x < \rho \}$, and $H_+(\rho, \theta) = \{x \in \RR^d \suchthat \theta \cdot x > \rho \}$.
Let $f_X(\cdot)$ denote the density function of random variable $X$.
We will use the affine Blaschke-Petkantschin formula as stated in \cite[Theorem 7.2.7]{schneider}.
Let $c_d = \details{(b_{d, d-1} (d-1)!)^2 (2/\omega_d)^2=} (d-1)!^2$\details{factor $(2/\omega_d)^2$ is to go from my hyperspherical coordinates to their normalization in page 168}.
\details{Note that the next part only uses independence of $X$ and $Y$: $X, Y$ can be two independent sets of $d$ random points in $\RR^d$ with joint density $f_X(x) f_Y(y)$.}
Shorthand notation $\vol{x}$ denotes the $(d-1)$-dimensional volume of the simplex determined by the $d$ points in $d$-tuple or matrix with $d$ columns $x$. 
Lowercase $x$ is a matrix and an integration variable and it represents a particular value of random variable $X$.
We have
\begin{align*}
\pr\bigl(&X, Y \in F([X,Y])\bigr) \\
&= \int_{\RR^{d^2}} \int_{\RR^{d^2}} \ind\{x \in F([x,y])\} \ind\{y \in F([x,y])\} f_X(x) f_Y(y) \ud x \ud y \\
&= c_d\int_{\RR_+^2} \int_{{(S^{d-1})}^2} \int_{{H(\rho_1, \theta_1)}^d} \int_{H(\rho_2, \theta_2)^d} \ind\{x \in F([x,y])\} \ind\{y \in F([x,y])\} \\
&\qquad  \times \vol{x} \vol{y} f_X(x) f_Y(y) \ud y \ud x \ud \theta_1 \ud \theta_2 \ud \rho_1 \ud \rho_2 \\
%
%
&= c_d \int\limits_{\RR_+^2} \int\limits_{{(S^{d-1})}^2} \biggl(\int\limits_{{H(\rho_1, \theta_1)}^d} \ind\{x \in H_+(\rho_2, \theta_2)^d \cup H_-(\rho_2, \theta_2)^d \} \vol{x} f_X(x) \ud x \biggr) \\
&\qquad \biggl(\int_{{H(\rho_2, \theta_2)}^d} \ind\{y \in H_+(\rho_1, \theta_1)^d \cup H_-(\rho_1, \theta_1)^d \} \vol{y} f_Y(y) \ud y \biggr) \numberthis \label{eq:lowerboundstart} \\
&\qquad \ud \theta_1 \ud \theta_2 \ud \rho_1 \ud \rho_2. 
\end{align*}

\paragraph{Upper bound.} For the upper bound we continue from \eqref{eq:lowerboundstart} as follows:
\begin{equation}\label{equ:sum}
\begin{aligned}
\pr\bigl(&X, Y \in F([X,Y])\bigr) \\
&= c_d \sum_{s, s' \in \{-,+\}} \int_{\RR_+^2} \int_{{(S^{d-1})}^2} 
\biggl(\int_{{H(\rho_1, \theta_1)}^d} \ind\{x \in H_{s}(\rho_2, \theta_2)^d\} \vol{x} f_X(x) \ud x \biggr) \\
&\qquad \biggl(\int_{{H(\rho_2, \theta_2)}^d} \ind\{y \in H_{s'}(\rho_1, \theta_1)^d\} \vol{y} f_Y(y) \ud y \biggr) 
\ud \theta_1 \ud \theta_2 \ud \rho_1 \ud \rho_2.
\end{aligned}
\end{equation}

For the next step we will need the following notation:  
$Z= (Z_1, \dotsc, Z_d) \in \RR^{(d-1) \times d}$ is i.i.d.\ standard Gaussian (identifying $H(\rho_1, \theta_1)$ with $\RR^{d-1}$). 
Also, $h_s(\rho_1, \theta_1, \rho_2, \theta_2)$ for $s \in \{+,-\}$ is the halfspace $H_s(\rho_2, \theta_2) \cap H(\rho_1, \theta_1)$ in $\RR^{d-1}$ (identifying $H(\rho_1, \theta_1)$ with $\RR^{d-1}$, see \cref{fig:halfspaces}).
Finally, $E$ is the event $\{ Z \in h_-(\rho_1, \theta_1, \rho_2, \theta_2)^d \}$,
and $\mu$ is the Gaussian probability measure in $\RR^{d-1}$.
\begin{figure}
\begin{center}
\resizebox{.5\columnwidth}{!}{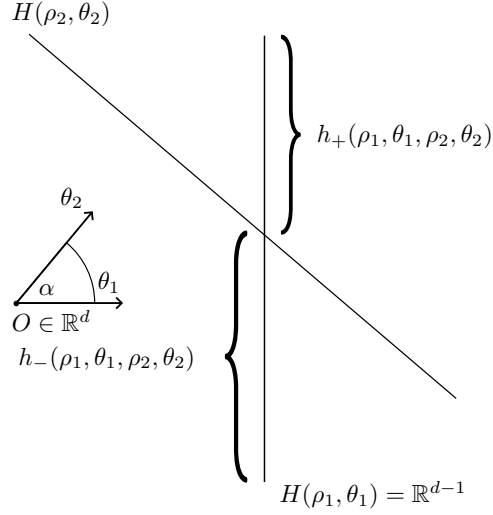}
\caption{Halfspaces in the proof of \cref{lem:pair}}\label{fig:halfspaces}
\end{center}
\end{figure}

We have
\begin{align*}
\int_{{H(\rho_1, \theta_1)}^d} &\ind\{x \in H_-(\rho_2, \theta_2)^d \} \vol{x} f_X(x) \ud x \\
&= \left(\int_{{H(\rho_1, \theta_1)}^d} f_X(x) \ud x\right) \e_Z\bigl(\vol{Z} \ind\event \bigr) \\
&= \left(\int_{{H(\rho_1, \theta_1)}^d} f_X(x) \ud x\right) \pr_Z\bigl(\event\bigr) \e_Z\bigl( \vol{Z} \bigm| \event \bigr) \\
&= \frac{e^{-d\rho_1^2/2}}{(2\pi)^{d/2}} \bigl(\mu(h_-(\rho_1, \theta_1, \rho_2, \theta_2))\bigr)^d \e_Z \bigl( \vol{Z} \bigm| \event \bigr).
\end{align*}
Let $A$ be the covariance matrix of the Gaussian distribution truncated to $h_-(\rho_1, \theta_1, \rho_2, \theta_2)$. 
Namely, $A = \cov\bigl(Z_1 \giventhat Z_1 \in h_-(\rho_1, \theta_1, \rho_2, \theta_2) \bigr)$.
Note that the variance of any univariate marginal of $Z_1$ conditioned on $Z_1 \in h_-(\rho_1, \theta_1, \rho_2, \theta_2)$ is at most 1 (say, by the Brascamp-Lieb inequality \cite[Section 5]{BRASCAMP1976366}) and this implies $\det A \leq 1$. 
Using moment inequalities and \cref{prop:blaschke} (Blaschke's formula):
\begin{align*}
\e_Z\bigl( \vol{Z} \giventhat \event \bigr) 
&\leq \sqrt{ \e_Z\bigl( \vol{Z}^2 \bigm| \event \bigr) } 
= \sqrt{ \frac{d}{(d-1)!} \det A} 
\leq \sqrt{\frac{d}{(d-1)!}}.
\end{align*}

Now, to express the Gaussian measure of $h_-(\rho_1, \theta_1, \rho_2, \theta_2)$, we need the signed distance of its boundary to the origin of $\RR^{d-1}$ (the sign is positive if the halfspace contains the origin). 
The signed distance is $t = t(\rho_1, \theta_1, \rho_2, \theta_2) = \frac{\rho_2-\rho_1 \cos \alpha}{\sin \alpha}$, where $\alpha \in [0,\pi]$ is the angle between $\theta_1$ and $\theta_2$ (see \cref{fig:halfspaces}).\footnote{To see this, note first that it is enough to peform this calculation in $\RR^2$. Assume without loss of generality that $\theta_1=(1,0)$ and $\theta_2= (\cos \alpha, \sin \alpha)$. Then $t$ is the $y$-coordinate of the point intersection of the lines $(x,y)\cdot \theta_1 = \rho_1$ and $(x,y)\cdot \theta_2 = \rho_2$, which implies $x=\rho_1$ and $\rho_1 \cos \alpha + y \sin \alpha = \rho_2$. The claim follows.}  
In other words, $t = \frac{\rho_2-\rho_1 \inner{\theta_1}{\theta_2}}{\sqrt{1-(\inner{\theta_1}{\theta_2})^2}}$.
To understand this quantity, it will be helpful in the next calculation to reinterpret certain integrals as expectations and to think of $\theta_1$ and $\theta_2$ as random unit vectors.
With that interpretation, we will use the following fact: the distribution of $W := \inner{\theta_1}{\theta_2}$ has density $w \mapsto \frac{\Gamma(\frac{d}{2})}{\sqrt{\pi}\Gamma(\frac{d-1}{2})} (1-w^2)^{\frac{d-3}{2}}$ with support $[-1,1]$.\footnote{To see this, without loss of generality we can assume that $\theta_2=e_1$. Then use Archimedes' idea, namely that the distribution of the first $d-2$ coordinates of $\theta_1$ is uniform in the unit $(d-2)$-dimensional ball. 
The claim follows then up to the normalization constant. 
The constant can be obtained by integration.}

Let $\omega_d = 2 \pi^{d/2}/\Gamma(d/2)$ be the area of the unit sphere in $\RR^d$. 
We determine the asymptotics of the first term in the sum in \cref{equ:sum} using \cref{prop:asymptotic} in the last step:
\begin{multline}\label{eq:upperbound}
\begin{aligned}
c_d &\int_{\RR_+^2} \int_{{(S^{d-1})}^2} 
\biggl(\int_{{H(\rho_1, \theta_1)}^d} \ind\{x \in H_-(\rho_2, \theta_2)^d \} \vol{x} f_X(x) \ud x \biggr) \\
&\qquad \times \biggl(\int_{{H(\rho_2, \theta_2)}^d} \ind\{y \in H_-(\rho_1, \theta_1)^d \} \vol{y} f_Y(y) \ud y \biggr) \ud \theta_1 \ud \theta_2 \ud \rho_1 \ud \rho_2
\end{aligned}\\
\begin{aligned}
&\leq \frac{c_d d}{(d-1)!(2\pi)^d} \int_{\RR_+^2} \int_{{(S^{d-1})}^2} 
e^{-\frac{d(\rho_1^2+\rho_2^2)}{2}} \bigl(\mu(h_-(\rho_1, \theta_1, \rho_2, \theta_2))\bigr)^d \\
&\qquad \times \bigl(\mu(h_-(\rho_2, \theta_2, \rho_1, \theta_1))\bigr)^d \ud \theta_1 \ud \theta_2 \ud \rho_1 \ud \rho_2 \\
%
%
%
&= \frac{d! \omega_d^2}{(2\pi)^d} \int_{\RR_+^2} e^{-\frac{d(\rho_1^2+\rho_2^2)}{2}} \e_{\theta_1, \theta_2}\Bigl( 
\bigl(\Phi(t(\rho_1, \theta_1, \rho_2, \theta_2))\bigr)^d   \\
&\qquad \times \bigl(\Phi(t(\rho_2, \theta_2, \rho_1, \theta_1))\bigr)^d \Bigr) \ud \rho_1 \ud \rho_2 \\
&= \frac{d! \omega_d^2}{(2\pi)^d} \int\limits_{\RR_+^2} e^{-\frac{d(\rho_1^2+\rho_2^2)}{2}} 
\e_W\biggl( \left(\Phi\Bigl(\frac{\rho_2-\rho_1 W}{\sqrt{1-W^2}}\Bigr)  
\Phi\Bigl(\frac{\rho_1-\rho_2 W}{\sqrt{1-W^2}}\Bigr)\right)^d \biggr) \ud \rho_1 \ud \rho_2 \\
&= \frac{d! \omega_d^2 \Gamma(\frac{d}{2})}{(2\pi)^d\sqrt{\pi}\Gamma(\frac{d-1}{2})} 
\int\limits_{\RR_+^2} e^{-\frac{d(\rho_1^2+\rho_2^2)}{2}} 
\int\limits_{-1}^1 \left(\Phi\left(\frac{\rho_2-\rho_1 w}{\sqrt{1-w^2}}\right)\Phi\left(\frac{\rho_1-\rho_2 w}{\sqrt{1-w^2}}\right)\right)^d  \\
&\qquad \times (1-w^2)^{\frac{d-3}{2}} \ud w \ud \rho_1 \ud \rho_2 \\
%
&\leq 2^{o(d)} \!\!\! \int\limits_{\RR_+^2} \int\limits_{-1}^1 
 \!\left( e^{\frac{-\rho_1^2-\rho_2^2}{2}}\Phi\Bigl(\frac{\rho_2-\rho_1 w}{\sqrt{1-w^2}}\Bigr)\Phi\Bigl(\frac{\rho_1-\rho_2 w}{\sqrt{1-w^2}}\Bigr)\sqrt{1-w^2}\right)^{d-3} 
\!\!\!\ud w \ud \rho_1 \ud \rho_2 \\
&= C_{\ref*{lem:pair}}^{d+o(d)},
\end{aligned}
\end{multline}
where
\begin{equation}\label{eq:constant}
C_{\ref*{lem:pair}} :=
\sup_{\substack{\rho_1, \rho_2 \geq 0 \\ w \in [-1,1]}} e^{-\frac{\rho_1^2+\rho_2^2}{2}}\Phi\left(\frac{\rho_2-\rho_1 w}{\sqrt{1-w^2}}\right)\Phi\left(\frac{\rho_1-\rho_2 w}{\sqrt{1-w^2}}\right) \sqrt{1-w^2} \approx 0.4424.
\end{equation}
\details{estimating sup with Mathematica.}
\details{
Similarly,
\begin{multline*}
\begin{aligned}
c_d\int_{\RR_+^2} &\int_{{(S^{d-1})}^2} 
\biggl(\int_{{H(\rho_1, \theta_1)}^d} \ind\{x \in H_+(\rho_2, \theta_2)^d \} \vol{x} f_X(x) dx \biggr) \\
&\qquad \biggl(\int_{{H(\rho_2, \theta_2)}^d} \ind\{y \in H_-(\rho_1, \theta_1)^d \} \vol{y} f_Y(y) dy \biggr) d \theta_1 d \theta_2 d \rho_1 d \rho_2
\end{aligned}\\
\begin{aligned}
&= (
\sup_{\rho_1, \rho_2 \geq 0, w \in [-1,1]} e^{-(\rho_1^2+\rho_2^2)/2}(1-\Phi(\frac{\rho_2-\rho_1 w}{\sqrt{1-w^2}}))\Phi(\frac{\rho_1-\rho_2 w}{\sqrt{1-w^2}}) \sqrt{1-w^2}
)^{d+o(d)} 
\end{aligned}
\end{multline*}

\begin{multline*}
\begin{aligned}
c_d\int_{\RR_+^2} &\int_{{(S^{d-1})}^2} 
\biggl(\int_{{H(\rho_1, \theta_1)}^d} \ind\{x \in H_+(\rho_2, \theta_2)^d \} \vol{x} f_X(x) dx \biggr) \\
&\qquad \biggl(\int_{{H(\rho_2, \theta_2)}^d} \ind\{y \in H_+(\rho_1, \theta_1)^d \} \vol{y} f_Y(y) dy \biggr) d \theta_1 d \theta_2 d \rho_1 d \rho_2
\end{aligned}\\
\begin{aligned}
&= (
\sup_{\rho_1, \rho_2 \geq 0, w \in [-1,1]} e^{-(\rho_1^2+\rho_2^2)/2}(1-\Phi(\frac{\rho_2-\rho_1 w}{\sqrt{1-w^2}}))(1-\Phi(\frac{\rho_1-\rho_2 w}{\sqrt{1-w^2}})) \sqrt{1-w^2}
)^{d+o(d)} 
\end{aligned}
\end{multline*}
}
The other three terms in \cref{equ:sum} have similar asymptotics, with $C_{\ref*{lem:pair}}$ replaced by 
\[
\sup_{\substack{\rho_1, \rho_2 \geq 0 \\ w \in [-1,1]}} e^{-\frac{\rho_1^2+\rho_2^2}{2}}\left(1-\Phi\left(\frac{\rho_2-\rho_1 w}{\sqrt{1-w^2}}\right)\right)\Phi\left(\frac{\rho_1-\rho_2 w}{\sqrt{1-w^2}}\right) \sqrt{1-w^2} \approx 0.355
\]
and
\[
\sup_{\substack{\rho_1, \rho_2 \geq 0 \\ w \in [-1,1]}} e^{-\frac{\rho_1^2+\rho_2^2}{2}}\left(1-\Phi\left(\frac{\rho_2-\rho_1 w}{\sqrt{1-w^2}}\right)\right)\left(1-\Phi\left(\frac{\rho_1-\rho_2 w}{\sqrt{1-w^2}}\right)\right) \sqrt{1-w^2} = 1/4.
\]
Namely, the first of the four terms in \cref{equ:sum} is asymptotically the largest and we have:
\[
\pr\bigl(X, Y \in F([X,Y])\bigr) \leq C_{\ref*{lem:pair}}^{d+o(d)}.
\]
Finally, note that the argument of $\sup$ in \cref{eq:constant} is logconcave and symmetric in $\rho_1, \rho_2$ for any fixed $w$ (using the known fact that $\Phi$ is logconcave).
This implies that its value at $\rho_1, \rho_2,w$ is less than or equal to its value at $(\rho_1+ \rho_2)/2, (\rho_1+ \rho_2)/2,w$ and therefore it is enough to maximize for $\rho_1 = \rho_2$ and we have the simplified expression in the statement of the theorem.

\paragraph{Lower bound.} In \cref{eq:lowerboundstart}, consider the term 
\[
\ind\{x \in H_+(\rho_2, \theta_2)^d \cup H_-(\rho_2, \theta_2)^d \}.
\] 
Note that (a.s.) one of $H_+(\rho_2, \theta_2)$ and $H_-(\rho_2, \theta_2)$ is the ``biggest'' in the particular sense that it contains in its interior the point in $H(\rho_1, \theta_1)$ (the domain of the innermost integral) that is closest to the origin (namely point $\rho_1 \theta_1$).
More precisely, let $H_M(\rho_1, \theta_1,\rho_2, \theta_2)$ be (a.s.) the halfspace among $H_+(\rho_2, \theta_2)$ and $H_-(\rho_2, \theta_2)$ that contains $\rho_1 \theta_1$ in its interior.
Then
\begin{equation}\label{eq:indicators}
\ind\{x \in H_+(\rho_2, \theta_2)^d \cup H_-(\rho_2, \theta_2)^d \} \geq \ind\{x \in H_M(\rho_1, \theta_1,\rho_2, \theta_2)^d \}.
\end{equation}

Let $Z= (Z_1, \dotsc, Z_d) \in \RR^{(d-1) \times d}$ be i.i.d.\ standard Gaussian (identifying $H(\rho_1, \theta_1)$ with $\RR^{d-1}$), let $E'$ be the event $\{ Z \in h_M(\rho_1, \theta_1, \rho_2, \theta_2)^d \}$, and let $h_{M}(\rho_1, \theta_1, \rho_2, \theta_2)$ be the halfspace $H_M(\rho_1, \theta_1,\rho_2, \theta_2) \cap H(\rho_1, \theta_1)$ in $\RR^{d-1}$ (identifying $H(\rho_1, \theta_1)$ with $\RR^{d-1}$).
Now, using \cref{lem:lowerbound} (a lower bound on the expected volume of a random simplex in a halfspace), we have
\begin{equation}\label{eq:integraltoexpectation}
\begin{aligned}
\int_{{H(\rho_1, \theta_1)}^d} &\ind\{x \in H_M(\rho_1, \theta_1,\rho_2, \theta_2)^d \} \vol{x} f_X(x) \ud x \\
&= \e_Z (\vol{Z}\ind E' ) \int_{{H(\rho_1, \theta_1)}^d} f_X(x) \ud x   \\
&= \pr_Z ( E' ) \e_Z\bigl( \vol{Z} \bigm| E' \bigr) \int_{{H(\rho_1, \theta_1)}^d} f_X(x) \ud x \\
&= \frac{e^{-d\rho_1^2/2}}{(2\pi)^{d/2}} \bigl(\mu(h_{M}(\rho_1, \theta_1, \rho_2, \theta_2))\bigr)^d \e_Z\bigl( \vol{Z} \bigm| E' \bigr) \\
&\geq 2^{o(d)} (e/d)^{d/2} \frac{e^{-d\rho_1^2/2}}{(2\pi)^{d/2}} \bigl(\mu(h_{M}(\rho_1, \theta_1, \rho_2, \theta_2))\bigr)^d \\
&\geq 2^{o(d)} (e/d)^{d/2} \frac{e^{-d\rho_1^2/2}}{(2\pi)^{d/2}} \bigl(\mu(h_{-}(\rho_1, \theta_1, \rho_2, \theta_2))\bigr)^d.
\end{aligned}
\end{equation}
%
%
Using a calculation similar to \cref{eq:upperbound} but starting at \cref{eq:lowerboundstart} and using \cref{eq:indicators,eq:integraltoexpectation} twice we get
\begin{align*}
&\pr\bigl(X, Y \in F([X,Y])\bigr) \\
&= c_d \int\limits_{\RR_+^2} \int\limits_{{(S^{d-1})}^2} 
\biggl(\int\limits_{{H(\rho_1, \theta_1)}^d} \ind\{x \in H_+(\rho_2, \theta_2)^d \cup H_-(\rho_2, \theta_2)^d \} \vol{x} f_X(x) \ud x \biggr) \\
&\quad \biggl(\int\limits_{{H(\rho_2, \theta_2)}^d} \ind\{y \in H_+(\rho_1, \theta_1)^d \cup H_-(\rho_1, \theta_1)^d \} \vol{y} f_Y(y) \ud y \biggr) 
\ud \theta_1 \ud \theta_2 \ud \rho_1 \ud \rho_2 \\
&\geq c_d 2^{o(d)} \left(\frac{e}{2\pi d}\right)^d \int_{\RR_+^2} \int_{{(S^{d-1})}^2} e^{-d(\rho_1^2 + \rho_2^2)/2} \\
&\quad \bigl(\mu(h_{-}(\rho_1, \theta_1, \rho_2, \theta_2)) \mu(h_{-}(\rho_2, \theta_2, \rho_1, \theta_1))\bigr)^d \ud \theta_1 \ud \theta_2 \ud \rho_1 \ud \rho_2 \\
&= C_{\ref*{lem:pair}}^{d+o(d)}.\qedhere
\end{align*}
\end{proof}

\begin{proof}[Proof of \cref{thm:estranged}]
Immediate from \cref{lem:pair} and the fact that the number of $d$-subsets of $X$ is $\binom{2d}{d} = 4^{d+o(d)}$.
\end{proof}

\paragraph{Acknowledgments.}
We would like to thank K\'aroly J.\ B\"or\"oczky and Daniel Hug for helpful discussions.
This material is based upon work supported by the National Science Foundation under Grants CCF-1657939, CCF-1934568 and CCF-2006994.
This material is also based upon work supported by the National Science Foundation under Grant No. DMS-1929284 while the second author was in residence at the Institute for Computational and Experimental Research in Mathematics in Providence, RI, during the ``Harmonic Analysis and Convexity'' program.

\bibliographystyle{abbrv}
\bibliography{bib}

\begin{thebibliography}{10}

\bibitem{MR1149653}
F.~Affentranger and R.~Schneider.
\newblock Random projections of regular simplices.
\newblock {\em Discrete Comput. Geom.}, 7(3):219--226, 1992.

\bibitem{BaranySteiger}
I.~B\'{a}r\'{a}ny and W.~Steiger.
\newblock On the expected number of {$k$}-sets.
\newblock {\em Discrete Comput. Geom.}, 11(3):243--263, 1994.

\bibitem{BVV}
I.~B\'{a}r\'{a}ny, S.~Vempala, and A.~Vetta.
\newblock Nash equilibria in random games.
\newblock {\em Random Structures Algorithms}, 31(4):391--405, 2007.

\bibitem{MR1254086}
Y.~M. Baryshnikov and R.~A. Vitale.
\newblock Regular simplices and {G}aussian samples.
\newblock {\em Discrete Comput. Geom.}, 11(2):141--147, 1994.

\bibitem{spherical}
G.~Bonnet and E.~O'Reilly.
\newblock Facets of spherical random polytopes.
\newblock {\em Math. Nachr.}, 295(10):1901--1933, 2022.

\bibitem{BLR}
K.~J. B\"{o}r\"{o}czky, G.~Lugosi, and M.~Reitzner.
\newblock Facets of high-dimensional {G}aussian polytopes, 2018.

\bibitem{BRASCAMP1976366}
H.~J. Brascamp and E.~H. Lieb.
\newblock On extensions of the {B}runn-{M}inkowski and {P}rékopa-{L}eindler
  theorems, including inequalities for log concave functions, and with an
  application to the diffusion equation.
\newblock {\em Journal of Functional Analysis}, 22(4):366--389, 1976.

\bibitem{MR3185453}
S.~Brazitikos, A.~Giannopoulos, P.~Valettas, and B.-H. Vritsiou.
\newblock {\em Geometry of isotropic convex bodies}, volume 196 of {\em
  Mathematical Surveys and Monographs}.
\newblock American Mathematical Society, Providence, RI, 2014.

\bibitem{MR1698257}
D.~Bremner and V.~Klee.
\newblock Inner diagonals of convex polytopes.
\newblock {\em J. Combin. Theory Ser. A}, 87(1):175--197, 1999.

\bibitem{clarkson}
K.~L. Clarkson.
\newblock On the expected number of $k$-sets of coordinate-wise independent
  points.
\newblock \url{https://kenclarkson.org/cwi_ksets/p.pdf}, 2004.
\newblock Manuscript.

\bibitem{MR0363986}
P.~Erd\H{o}s, L.~Lov\'{a}sz, A.~Simmons, and E.~G. Straus.
\newblock Dissection graphs of planar point sets.
\newblock In {\em A survey of combinatorial theory}, pages 139--149, 1973.

\bibitem{MR2875755}
B.~Fleury.
\newblock Poincar\'{e} inequality in mean value for {G}aussian polytopes.
\newblock {\em Probab. Theory Related Fields}, 152(1-2):141--178, 2012.

\bibitem{HMR}
D.~Hug, G.~O. Munsonius, and M.~Reitzner.
\newblock Asymptotic mean values of {G}aussian polytopes.
\newblock {\em Beitr\"{a}ge Algebra Geom.}, 45(2):531--548, 2004.

\bibitem{HR}
D.~Hug and M.~Reitzner.
\newblock Gaussian polytopes: variances and limit theorems.
\newblock {\em Adv. in Appl. Probab.}, 37(2):297--320, 2005.

\bibitem{MR206823}
V.~Klee and D.~W. Walkup.
\newblock The {$d$}-step conjecture for polyhedra of dimension {$d<6$}.
\newblock {\em Acta Math.}, 117:53--78, 1967.

\bibitem{MR4376581}
B.~Leroux and L.~Rademacher.
\newblock Improved bounds for the expected number of $k$-sets.
\newblock {\em Discrete Comput. Geom.}, 2023.

\bibitem{lovasz1971number}
L.~Lov{\'a}sz.
\newblock On the number of halving lines.
\newblock {\em Ann. Univ. Sci. Budapest, Etovos, Sect. Math.}, 14:107--108,
  1971.

\bibitem{LV07}
L.~Lov{\'a}sz and S.~Vempala.
\newblock The geometry of logconcave functions and sampling algorithms.
\newblock {\em Random Struct. Algorithms}, 30(3):307--358, 2007.

\bibitem{10.2307/1426176}
R.~E. Miles.
\newblock Isotropic random simplices.
\newblock {\em Advances in Applied Probability}, 3(2):353--382, 1971.

\bibitem{MR2891161}
L.~Rademacher.
\newblock On the monotonicity of the expected volume of a random simplex.
\newblock {\em Mathematika}, 58(1):77--91, 2012.

\bibitem{MR258089}
H.~Raynaud.
\newblock Sur l'enveloppe convexe des nuages de points al\'{e}atoires dans
  {$R^{n}$}. {I}.
\newblock {\em J. Appl. Probability}, 7:35--48, 1970.

\bibitem{MR156262}
A.~R\'{e}nyi and R.~Sulanke.
\newblock \"{U}ber die konvexe {H}\"{u}lle von {$n$} zuf\"{a}llig gew\"{a}hlten
  {P}unkten.
\newblock {\em Z. Wahrscheinlichkeitstheorie und Verw. Gebiete}, 2:75--84
  (1963), 1963.

\bibitem{MR924157}
W.~Rudin.
\newblock {\em Real and complex analysis}.
\newblock McGraw-Hill Book Co., New York, third edition, 1987.

\bibitem{schneider}
R.~Schneider and W.~Weil.
\newblock {\em Stochastic and integral geometry}.
\newblock Probability and its Applications (New York). Springer-Verlag, Berlin,
  2008.

\bibitem{Vs}
B.~von Stengel.
\newblock New maximal numbers of equilibria in bimatrix games.
\newblock {\em Discrete Comput. Geom.}, 21(4):557--568, 1999.

\bibitem{WagnerSurvey}
U.~Wagner.
\newblock {$k$}-sets and {$k$}-facets.
\newblock In {\em Surveys on discrete and computational geometry}, volume 453
  of {\em Contemp. Math.}, pages 443--513. Amer. Math. Soc., Providence, RI,
  2008.

\end{thebibliography}

\end{document}